\documentclass[sn-mathphys-num]{sn-jnl}

\usepackage{graphicx}%
\usepackage{multirow}%
\usepackage[title]{appendix}%
\usepackage{xcolor}%
\usepackage{listings}%
\usepackage{natbib}
\setcitestyle{square, numbers}

\usepackage{indentfirst,enumerate,cite,amssymb,amsfonts,amsmath,amsthm,mathrsfs,dsfont}
\usepackage{color}
\usepackage{multicol}
\usepackage[utf8]{inputenc}
\usepackage{ulem}
\newtheorem{theorem}{Theorem}[section]
\newcommand{\R}{\mathbb{R}}

\begin{document}
	
	\title[Boundedness of vector fields on compact Lie groups]{On the Sobolev boundedness of vector fields on compact Riemannian manifolds}


	\author[1]{\fnm{Duván} \sur{Cardona}}\email{ Duvan.CardonaSanchez@UGent.be}
	\equalcont{These authors contributed equally to this work.}
	
	\author[2]{\fnm{Wagner A. A.} \sur{de Moraes}}\email{wagnermoraes@ufpr.br}
	\equalcont{These authors contributed equally to this work.}
	
	\author[2]{\fnm{André} \sur{Pedroso Kowacs}}\email{andrekowacs@gmail.com}
	\equalcont{These authors contributed equally to this work.}
	
	\author*[2]{\fnm{Alexandre} \sur{Kirilov}}\email{akirilov@ufpr.br}
	\equalcont{These authors contributed equally to this work.}
	
	\affil[1]{\orgdiv{Department of Mathematics: Analysis, Logic and Discrete Mathematics; Ghent Analysis \& PDE Center}, \orgname{Ghent University}, \orgaddress{\street{ \ Krijgslaan 281, Building S8}, \city{Ghent}, \postcode{9000}, \country{Belgium}}}
	
	\affil[2]{\orgdiv{Departamento de Matem\'atica}, \orgname{Universidade Federal do Paran\'a}, \orgaddress{\street{ \\ Caixa Postal 19096}, \city{Curitiba}, \postcode{81531-990}, \country{Brasil}}}


	\abstract{We analyze the sharpness of the Sobolev order for left-invariant vector fields on compact Riemannian manifolds. Utilizing techniques from pseudo-differential operator theory and microlocal analysis, we investigate the asymptotic behavior of eigenvalues associated with these vector fields. As an application, we demonstrate the ill-posedness of a class of Cauchy problems involving left-invariant vector fields on compact Lie groups.}

	
	
	\maketitle
	
	\section{Introduction} 

	The study of global properties of vector fields defined on compact manifolds has attracted the attention of the scientific community for over half a century. Some contributions to this field can be found in \citep{Berga2,Berga3,Petro}. Upon exploring these and related references, it becomes clear that the focus is on the study of global solvability and global hypoellipticity when the smooth manifold under consideration is a torus. This emphasis arises from two motivations: Firstly, the conjecture by S. Greenfield and N. Wallach proposes that globally hypoelliptic vector fields do not exist on compact manifolds that are not diffeomorphic to tori \citep{Forni,GW73}. Secondly, over the past three decades, a robust framework of techniques has been developed specifically for studying vector fields on tori. These techniques rely mainly on the good properties of the Fourier series, the analysis of the asymptotic behavior of eigenvalues of vector fields, and incorporate \textit{a priori} estimates and Diophantine conditions that naturally emerge in this context \citep{GW73,Petro}.
	
	Recent advances in the quantization framework introduced by Ruzhansky and Turunen in \citep{Ruz} have allowed the extension of many of these results from tori to general compact Lie groups. 
	This progress was made possible through the development of new techniques, which are outlined in \citep{Kir1, Kir2, Kir3, Kow, RuzTurWir, RuzhanskyWirth2015}. During their exploration of this new context, the authors came across a relevant question that had not yet been addressed in the literature. This question posed a challenge to the further development of the theory and can be stated as follows:
	
	\begin{quotation}
		\noindent {\it ``Can a left-invariant nonzero vector field $X$ on $G$ be bounded from $L^2_s(G)$ to $L^2_{s-\varepsilon}(G)$, for some $0 < \varepsilon < 1$?''}
	\end{quotation}
	
	Here, $(G,g)$ denotes a compact Riemannian manifold with a metric $g$. For $s \in \mathbb{R}$, the Sobolev space $L^2_s(G)$ is defined as the completion of $C^\infty(G)$ with respect to the norm
	\begin{equation}\label{Sobolev_norm}
		\| f \|_{L^2_s} := \| (\operatorname{Id}-\Delta_G)^{\frac{s}{2}} f \|_{L^2(G)},
	\end{equation}
	where $-\Delta_G$ is the (positive) Laplacian associated with the Riemannian metric $g$.
	
	In \citep{Ruz} the authors proved that any vector field $X$ induces a bounded operator $X : L^2_s(G) \to L^2_{s-1}(G)$, therefore the above question investigates the sharpness of this continuity property.	
	
	In this paper, we provide a negative answer to this question by proving the following theorem:
	
	\begin{theorem}\label{Lemma} 
		Let $(G,g)$ denote a compact Riemannian manifold.
		For $0\leq \varepsilon<1$ and $s\in \mathbb{R}$, any smooth vector field $X\in TG\setminus \{0\}$ does not represent a bounded operator from $L^2_s(G)$ to $L^2_{s-\varepsilon}(G)$.
	\end{theorem}
	
	The proof of this result, presented in Section \ref{Proof:mth:2}, employs techniques and results from pseudo-differential theory on compact manifolds. 
	
	A direct consequence of this theorem is that the sequence of eigenvalues associated with a left-invariant vector field on a compact Lie group has a subsequence that grows at a polynomial rate as the index tends to infinity. In Section \ref{application}, we apply this result to demonstrate that a class of Cauchy problems on a compact Lie group is ill-posed in $[0,T] \times G$.
	
	{\noindent\bf{Notation.}} 
	The notation $A\asymp B,$ indicate that there exist constants $C_1, C_2 > 0$, independent of the fundamental parameters, such that $C_1 B \leq A \leq C_2 B$.

	\section{Sobolev boundedness of vector fields}\label{Proof:mth:2} 
	
	In this section, we prove our main result, Theorem \ref{Lemma}. To illustrate our approach, let us first consider the case where $\varepsilon = 0$. Note that the vector field $X$ can be realized as an unbounded linear operator on $L^2(G)$. The fact that $X$ is not bounded on $L^2(G)$ can be demonstrated as follows
	
	There exists a local coordinate system $(\omega, \phi)$ where $\omega \subset G$, and $\phi: \omega \to \phi(\omega) \subset \R^d$ such that  $0 \in \phi(\omega)$ and that the vector field $X$ is identified with one of the canonical vector fields $\partial_{x_j}$ of $\mathbb{R}^d$. Using the identification $\omega \cong \phi(\omega)$, we have $X \cong \partial_{x_j}$. 
	
	Define the sequence of compactly supported functions on $\phi(\omega)$ by 
	$$
	f_n(x) = e^{i nx_j} \psi_\omega(x), \quad n \in \mathbb{N},
	$$ 
	where $\psi_\omega$ is a smooth, positive, and compactly supported function on $\phi(\omega)$.

	Observe that 
	$$
	\|f_n \circ \phi\|_{L^2(\omega)} \asymp \|f_n\|_{L^2(\phi(\omega))} = \|\psi_\omega\|_{L^2(\phi(\omega))} <\infty.
	$$ 
	
	However,
	$$
	\|X (f_n \circ \phi)\|_{L^2(\omega)} \asymp \|X (f_n \circ \phi) \circ \phi^{-1}\|_{L^2(\phi(\omega))} \asymp \|\partial_{x_j} f_n\|_{L^2(\phi(\omega))},
	$$ 
	and
	$$
	\|\partial_{x_j} f_n\|_{L^2(\phi(\omega))} = \|e^{i n x_j} \partial_{x_j} \psi_\omega + i n e^{i n x_j} \psi_\omega\|_{L^2(\phi(\omega))}.
	$$ 
	
	Since 
	$$
	\left| e^{i nx_j} \partial_{x_j} \psi_\omega + i n e^{i nx_j} \psi_\omega \right| 
	\geq n |\psi_\omega| - |\partial_{x_j} \psi_\omega|,
	$$ 
	it follows that as $n \to \infty$, we have $\|X (f_n \circ \phi)\|_{L^2(\omega)} \to \infty$. 
	
	Consequently, since the sequence $\{f_{n} \circ \phi^{-1}\}$ remains bounded in $L^2(\phi(\omega))$, this establishes that $X$ is not a bounded operator on $L^2(\phi(\omega))$. Consequently, $X$ is not a bounded operator on $L^2(G)$.
	
	We generalize this idea to prove Theorem \ref{Lemma} for $0 \leq \varepsilon < 1$ by employing the pseudo-differential calculus on the manifold $G$ alongside the localization technique illustrated above. For the basic aspects of pseudo-differential operators, we refer the reader to Ruzhansky and Turunen \citep{Ruz}.
\begin{proof}[Proof of Theorem \ref{Lemma}]
		Assume that $0 \neq X : L^2_s(G) \to L^2_{s-\varepsilon}(G)$ is bounded, where $\varepsilon \in [0,1)$. In particular, this implies that 
		$$
		X : L^2(G) \to L^2_{-\varepsilon}(G)
		$$ 
		is bounded, meaning there exists a constant $C' > 0$ such that
		$$
		\| X f\|_{L^2_{-\varepsilon}}=\|(\operatorname{Id}-\Delta_G)^{-\frac{\varepsilon}{2}} X f\|_{L^2} \leq C' \|f\|_{L^2}.
		$$ 
        Let us choose an atlas  (consisting of a family of local coordinate systems) $(U_\omega,\phi_\omega)$ for $G,$ with the open sets $U_\omega$ providing a finite covering of $G.$ Let $\tilde{U}_\omega=\phi_\omega(U_\omega)$ be the image of the diffeomorphism $\phi_\omega: {U}_\omega\subset G\rightarrow \tilde{U}_\omega\subset  \mathbb{R}^n. $ 
        Consider the Euclidean Bessel potential $(\operatorname{Id}-\Delta_{\mathbb{R}^n})^{-\frac{\delta}{2}}$ defined by 
        $$(\operatorname{Id}-\Delta_{\mathbb{R}^n})^{-\frac{\delta}{2}} u(x)=\int e^{2\pi ix\cdot \xi}(1+\vert\xi\vert)^{-\frac{\delta}{2}}\widehat{u}(\xi)d\xi,\,\forall u\in C^\infty(\tilde{U}_\omega),\quad \,x\in \tilde{U}_\omega,$$ where $ \widehat{u} := \widehat{u}(\xi),$ denotes the Euclidean Fourier transform of $u$,  and $u$ is extended by zero outside of $\tilde U_\omega$, which we assume to be precompact. 
        
		The operator
        $(\operatorname{Id}-\Delta_{\mathbb{R}^n})^{-\frac{\delta}{2}} $ induces to an elliptic  operator $(\operatorname{Id}-\Delta)^{-\frac{\varepsilon}{2}}$ on $C^\infty(G),$ determined by its action on smooth functions of each local coordinate system via,
        \begin{equation}
           (\operatorname{Id}-\Delta)^{-\frac{\varepsilon}{2}}\phi:= (  (\operatorname{Id}-\Delta_{\mathbb{R}^n})^{-\frac{\delta}{2}} \phi\circ {\phi_\omega^{-1})\circ \phi_\omega},\forall \phi\in C^\infty(U_\omega).
        \end{equation} Then, this local definition of the operator can be extended to the whole manifold $G$ by using a partition of unity $\{\chi_\omega\}_\omega,$ and then, $\sum_\omega\chi_\omega=1,$ subordinated to the family of open sets $U_\omega,$ meaning that $\chi_\omega$ has its support contained in each $U_\omega.$

Due to the equivalent definitions of Sobolev spaces in compact manifolds, the Sobolev norms given by \eqref{Sobolev_norm} are equivalent to the ones given by 
    \begin{equation*}
        \|f\|'_{L^2_s}=\sum_{\tau=1}^m \|(\operatorname{Id}-\Delta_{\R^n})^{s/2} f_\tau\|_{L^2(\R^n)},
    \end{equation*}
    where $f_\tau = (\phi_\tau \cdot f)\circ \varphi_\tau^{-1}$, extended by zero outside $\tilde U_\tau=\varphi_\tau(U_\tau)$, $\{\varphi_\tau,U_\tau\}_{\tau=1}^m$ denotes a finite covering by local charts of $G$ and $\{\phi_\tau\}_{\tau=1}^m$ denotes a partition of unity subordinate to such covering. Hence $\|f\|_{L^2_\varepsilon(U)}\asymp \|(\operatorname{Id}-\Delta)^{\varepsilon/2}f\|_{L^2(U)}$. Therefore 
    \begin{align*}
       \| (\operatorname{Id}-\Delta)^{-\frac{\varepsilon}{2}}Xf\|_{L^2}\asymp \|Xf\|_{L^2_{-\varepsilon}}\leq C'\|f\|_{L^2}.
    \end{align*}

A consequence of this estimate is that $(\operatorname{Id}-\Delta)^{-\frac{\varepsilon}{2}} X: L^2(G)\rightarrow L^2(G)$ is bounded. Note also that
		$$
		R = X (\operatorname{Id}-\Delta)^{-\frac{\varepsilon}{2}} : L^2(G) \to L^2(G)
		$$         
is also a bounded operator. To prove this fact, note that
		\begin{align*}
			R &= X (\operatorname{Id}-\Delta)^{-\frac{\varepsilon}{2}} - (\operatorname{Id}-\Delta)^{-\frac{\varepsilon}{2}} X + (\operatorname{Id}-\Delta)^{-\frac{\varepsilon}{2}} X \\&= [X, (\operatorname{Id}-\Delta)^{-\frac{\varepsilon}{2}}] + (\operatorname{Id}-\Delta)^{-\frac{\varepsilon}{2}} X,
		\end{align*}
		where $[X, (\operatorname{Id}-\Delta)^{-\frac{\varepsilon}{2}}]$ denotes the commutator of $X$ and $(\operatorname{Id}-\Delta)^{-\frac{\varepsilon}{2}}$.

		Note that $(\operatorname{Id}-\Delta)^{-\frac{\varepsilon}{2}}$ is a pseudo-differential operator of order $-\varepsilon \leq 0$, which is a bounded operator on $L^2(G)$ by the Calder\'on-Vaillancourt theorem, see H\"ormander \citep{hormander1985analysis}. Furthermore, since the order of $[X, (\operatorname{Id}-\Delta_G)^{-\frac{\varepsilon}{2}}]$ is also $-\varepsilon<0$ it is also bounded on $L^2(G)$, and so we have that $R : L^2(G) \to L^2(G)$ is also a bounded operator. This means there exists a constant $C > 0$ such that:
		$$
		\|R f\|_{L^2} \leq C \|f\|_{L^2}.
		$$ Let us fix a chart $(U,\phi)=(U_{\omega_0},\phi_{\omega_0})$. Without loss of generality, assume that $y = 0$ is an interior point of $\tilde U=\phi(U)$ which we assume to have compact closure $\overline{\tilde U}$ in $\mathbb{R}^n$.  By applying a rotation if necessary, we may further assume that $X \cong \partial_{x_j} $ on $\tilde U$, corresponding to one of the partial derivatives with respect to the canonical coordinates.

		Let $\delta_0$ be the Dirac delta distribution at $0 \in \mathbb{R}^n$. Since $\delta_0\in H^{m}(\R^n)$ for $m<-n/2$, where $H^m(\R^n)$ denotes the classical Sobolev space of order $m$ in $\R^n$, and since $\operatorname{supp}(\delta)=\{0\}$,  we have that $(\operatorname{Id}-\Delta_{\mathbb{R}^n})^{-\frac{\delta}{2}}\delta_0\in L^2(\tilde U)$ for $\delta>n/2$.

         Since $(\operatorname{Id}-\Delta)^{-\frac{\delta}{2}}=  (\operatorname{Id}-\Delta_{\mathbb{R}^n})^{-\frac{\delta}{2}}$ in local coordinates, for $\delta>n/2$ we get that 
        \begin{align*}
        \|X(\operatorname{Id}-\Delta)^{-\frac{\varepsilon+\delta}{2}}\delta_0\circ\phi\|_{L^2(U)}&=\|X(\operatorname{Id}-\Delta)^{-\frac{\varepsilon}{2}}(\operatorname{Id}-\Delta)^{-\frac{\delta}{2}}\delta_0\circ\phi\|_{L^2(U)}\\
        &=\|R(\operatorname{Id}-\Delta)^{-\frac{\delta}{2}}\delta_0\circ\phi\|_{L^2(U)}\\
        &\leq C\|(\operatorname{Id}-\Delta)^{-\frac{\delta}{2}}\delta_0\circ\phi\|_{L^2(U)}\\&\asymp\|(\operatorname{Id}-\Delta_{\mathbb{R}^n})^{-\frac{\delta}{2}}\delta_0\|_{L^2(\tilde U)}<\infty.    
        \end{align*}

        Since $ 0 \leq \varepsilon < 1 $ and $ \delta > 0 $, we can choose
		\begin{equation*}
			\delta = \frac{n}{2} + \lambda,
		\end{equation*}
		where $ \lambda > 0 $ satisfies
		\begin{equation}\label{ref:e}
			0 < \lambda \leq 1 - \varepsilon.
		\end{equation}		
        Hence
\begin{align*}
 \infty>\Vert X (\operatorname{Id}-\Delta)^{-\frac{\varepsilon+\delta}{2}} \delta_0\circ \phi\Vert_{L^2( U)} &\asymp \Vert \partial_{x_j} (\operatorname{Id}-\Delta_{\R^n})^{-\frac{\varepsilon+\delta}{2}} \delta_0  \Vert_{L^2(\tilde U)}\\
 &= \Vert (\operatorname{Id}-\Delta_{\R^n})^{-\frac{\varepsilon+\delta}{2}}  \partial_{x_j} \delta_0   \Vert_{L^2(\tilde U)}\\
 &= \|\partial_{x_j}\delta_0\|_{H^{-(\varepsilon+\delta)}(\tilde U)}.
\end{align*}
Consequently, since $\partial_{x_j}\delta_0\in H^{-m}(\R^n)$ for $m>\frac{n}{2}+1$, we must satisfy the condition	\begin{equation}\label{L2:condition}
			\delta+\varepsilon > \frac{n}{2}+1.
		\end{equation}		
		However, observe that
		\begin{align*}
			\delta+\varepsilon=\frac{n}{2}+\lambda+\varepsilon\leq \frac{n}{2}+1.
		\end{align*}
		This contradicts \eqref{L2:condition}. Therefore, we conclude that there does not exist $ \varepsilon \in [0,1) $ such that $ X: L^2_s(G) \to L^2_{s-\varepsilon}(G) $ is bounded.
	\end{proof}

	\section{Application: Ill-posedness of a Cauchy problem associated with a left-invariant vector field}\label{application}
	In this section, we present an application of our main result, Theorem \ref{Lemma}. The notations used here are consistent with those introduced by Ruzhansky and Turunen \citep[Chapter 10]{Ruz}.
	
	Let $ G $ be a compact Lie group with normalized Haar measure $ dx $, and let $ \widehat{G} $ denote its unitary dual (the set of equivalence classes of irreducible continuous unitary representations of $ G $). Consider a left-invariant vector field $ X \neq 0 $ on $ G $. We will use Theorem \ref{Lemma} to establish the following result.
	
	\begin{theorem}\label{Th.2}
		For any $T>0$ there exist initial data $u_0\in L^2(G)$ such that the Cauchy problem
		\begin{equation}\label{eqcauchyproblem}
			\begin{cases}
				\partial_t u(t,x) + i X u(t,x) = 0, & (t,x) \in \R_+ \times G, \\
				u(0,x) = u_0(x), & x \in G,
			\end{cases}
		\end{equation}
		does not admit solution $u\in C([0,T],L^2_s(G))\cap C^1((0,T),L^2_s(G))$ for any $s\in\R$.
	\end{theorem}
	
	\begin{proof}
		Suppose there exists a solution $u\in C([0,T],L^2_s(G))\cap C^1((0,T),L^2_s(G))$  to \eqref{eqcauchyproblem}. Observe that $ i X $ is a symmetric operator on $ L^2(G) $. For any function $ f \in C^\infty(G) $, we can associate a matrix symbol to $ i X $ using the quantization formula:
		\begin{equation*}
			i X f(x) = \sum_{[\xi] \in \widehat{G}} \dim(\xi) \operatorname{Tr}\big[\xi(x) \sigma_{i X}(\xi) \widehat{f}(\xi)\big],
		\end{equation*}
		where $ \dim(\xi) $ denotes the dimension of the representation $ \xi $, and $ \widehat{f}(\xi) $ is the Fourier transform of $ f $ at $ \xi $, defined as
		\begin{equation*}
			\widehat{f}(\xi) = \int_G f(x) \xi(x)^{-1} \, dx.
		\end{equation*}
		
		The matrix symbol $ \sigma_{i X}(\xi) $ of $ i X $ is diagonalizable by unitary matrices. Thus, for each $ [\xi] \in \widehat{G} $, there exists a representative of $ [\xi] $, which we will also denote by $ \xi $, such that $ \sigma_{i X}(\xi) $ is diagonal. Specifically
		\begin{equation*}
			\sigma_{i X}(\xi)_{\alpha \beta} = \mu_\alpha(\xi) \delta_{\alpha \beta}, \quad 1 \leq \alpha, \beta \leq \dim(\xi),
		\end{equation*}
		where $ \delta_{\alpha \beta} $ is the Kronecker delta, and $ \mu_\alpha(\xi) $ are real eigenvalues associated with the action of $ i X $.

		Now, as $iX$ is not bounded from $L^2_s(G)$ to $L^2_{s-\frac{1}{2}}(G)$, there must exist a sequence of distinct terms $([\xi_n])_{n \in \mathbb{N}}$ in $\widehat{G}$ and a constant $C > 0$ such that
		\begin{equation*}
			|\mu_{\alpha(n)}(\xi_n)| \geq C\langle \xi_n\rangle^{\frac{1}{2}},
		\end{equation*}
		for some $1 \leq \alpha(n) \leq \dim(\xi_n)$ and every $n \in \mathbb{N}$. The quantity $\langle \xi \rangle$ refers to the standard weight associated with the representation $\xi$ and its definition in the context of Lie groups can be found in \citep[page 538]{Ruz}.

		Indeed, suppose this is not the case. Then, there would exist a sufficiently large constant $ C > 0 $ such that
		\begin{equation*}
			|\mu_{\alpha}(\xi)| \leq C\langle \xi\rangle^{\frac{1}{2}},
		\end{equation*}
		for every $ [\xi] \in \widehat{G} $ and $ 1 \leq \alpha \leq \dim(\xi) $. Then for $ v \in L^2_s(G)$, we have that
		\begin{align*}
			\|iXv\|_{L_{s-\frac{1}{2}}^2}^2 &= \sum_{[\xi] \in \widehat{G}} \dim(\xi) \sum_{1 \leq \alpha, \beta \leq \dim(\xi)} \langle \xi\rangle^{2s-1} |\mu_{\alpha}(\xi)|^2 |\widehat{v}(\xi)_{\alpha\beta}|^2 \\
			&\leq C \sum_{[\xi] \in \widehat{G}} \dim(\xi) \sum_{1 \leq \alpha, \beta \leq \dim(\xi)} \langle \xi\rangle^{2s} |\widehat{v}(\xi)_{\alpha\beta}|^2 
			= \|v\|_{L_{s}^2}^2,
		\end{align*}
		implying that $ iX:L^2_s(G) \to L^2_{s-\frac{1}{2}}(G) $ is bounded, which contradicts our assumption.
		
		By composing $ \xi_n $ with a change of variables, we may further assume that $ \alpha(n) = 1 $ for all $ n \in \mathbb{N} $. Moreover, since $ \mu_{\alpha}(\overline{\xi}) = -\mu_\alpha(\xi) $, we may also assume that 
		\begin{equation}\label{ineqArchi}
			\mu_{1}(\xi_n) \geq C\langle \xi_n\rangle^{\frac{1}{2}},
		\end{equation}
		for all $ n \in \mathbb{N} $.
		
		Returning to the Cauchy problem, taking the group Fourier transform in $ G $ on \eqref{eqcauchyproblem} yields:
		\begin{equation*}
			\begin{cases}
				\partial_t \widehat{u}(t,\xi)_{\alpha\beta} + \mu_{\alpha}(\xi) \widehat{u}(t,\xi)_{\alpha\beta} = 0, & (t,\xi) \in [0,T] \times \widehat{G}, \\
				\widehat{u}(0,\xi)_{\alpha\beta} = \widehat{u_0}(\xi)_{\alpha\beta}, & \xi \in \widehat{G}, 
			\end{cases}\\[2mm]
		\end{equation*}
		for $ 1 \leq \alpha, \beta \leq \dim(\xi) $. 
		
		From this, it follows that
		\begin{equation}\label{eqwidehat}
			\widehat{u}(t,\xi)_{\alpha\beta} = e^{\mu_{\alpha}(\xi)t}\widehat{u_0}(\xi)_{\alpha\beta},
		\end{equation}
		for every $ \xi \in \widehat{G}, 1 \leq \alpha, \beta \leq \dim(\xi), t \in [0,T] $.
		
		Consider $ u_0 \in L^2(G) $ defined by:
		\begin{equation*}
			\widehat{u_0}(\xi_n)_{11} = e^{-\frac{CT}{2}\langle \xi_n\rangle^{\frac{1}{2}}},
		\end{equation*}
		for every $ n \in \mathbb{N} $, and $ \widehat{u_0}(\xi)_{\alpha\beta} = 0 $ otherwise. In fact, $ u_0 \in C^\infty(G) \subset L^2(G) $.
		
		From \eqref{eqwidehat} and \eqref{ineqArchi}, we conclude that
		\begin{equation*}
			|\widehat{u}(t,\xi_n)_{11}| \geq e^{C(t-\frac{T}{2})\langle\xi_n\rangle^{\frac{1}{2}}},
		\end{equation*}
		for every $ n \in \mathbb{N} $. Thus, $ u(t, \cdot) \not\in L^2_s(G)$ for $ t \geq \frac{T}{2} $ and any $s\in\R$, by Plancherel's identity, leading to a contradiction. We conclude that no solution $u$ exists in this case, proving the claim. \end{proof}

	\bmhead{Acknowledgements}
	Duván Cardona and André Kowacs were supported by the FWO Odysseus 1 grant G.0H94.18N: Analysis and Partial Differential Equations and by the Methusalem programme of the Ghent University Special Research Fund (BOF) (Grant number 01M01021). Duván Cardona was also supported by the Research Foundation-Flanders (FWO) under the postdoctoral grant No. 1204824N. André Kowacs was supported in part by the Coordenação de Aperfeiçoamento de Pessoal de Nível Superior - Brasil (CAPES) - Finance Code 001. Alexandre Kirilov and Wagner de Moraes were supported by the National Council for Scientific and Technological Development - CNPq, Brasil (grants 316850/2021-7 and 423458/2021-3).

\end{document}